\documentclass{amsart}

 \usepackage{amssymb}

\usepackage{amsthm}

\newtheorem{theorem}{Theorem}[section]
\newtheorem{lemma}{Lemma}[section]

\newtheorem{remark}{Remark}[section]

\numberwithin{equation}{section}

\newtheorem*{theorem*}{Theorem}

\newcommand{\R}{\mathbb{R}}
\newcommand{\N}{\mathbb{N}}
\newcommand{\be}{\begin{equation}}
\newcommand{\ee}{\end{equation}}

\begin{document}

\title[Multipliers of Laplace Transform Type for Laguerre and Hermite\dots]{Multipliers of Laplace Transform Type for Laguerre and Hermite Expansions}

\author{Pablo L.  De N\'apoli}
\address{Departamento de Matem\'atica \\
Facultad de Ciencias Exactas y Naturales \\ 
Universidad de Buenos Aires \\
Ciudad Universitaria \\
1428 Buenos Aires, Argentina}
\email{pdenapo@dm.uba.ar}

\author{Irene Drelichman}
\address{Departamento de Matem\'atica \\
Facultad de Ciencias Exactas y Naturales \\ 
Universidad de Buenos Aires \\
Ciudad Universitaria \\
1428 Buenos Aires, Argentina}
\email{irene@drelichman.com}

\author{Ricardo G. Dur\'an}
\address{Departamento de Matem\'atica \\
Facultad de Ciencias Exactas y Naturales \\ 
Universidad de Buenos Aires \\
Ciudad Universitaria \\
1428 Buenos Aires, Argentina}
\email{rduran@dm.uba.ar}

\thanks{Supported by ANPCyT under grant PICT 01307, by Universidad de
Buenos Aires under grants X070 and X837 and by CONICET under grants PIP 11420090100230 and PIP 11220090100625. The first and third authors are members of
CONICET, Argentina.}

\begin{abstract}
We present a new criterion for the weighted   $L^p-L^q$ boundedness  of multiplier operators for Laguerre and Hermite expansions that arise from a Laplace-Stieltjes transform. As a special case, we recover known results on weighted estimates for Laguerre and Hermite fractional integrals with a unified and simpler approach. 
\end{abstract}


\maketitle

\section{Introduction}

The aim of this paper is to obtain weighted estimates for multipliers of Laplace transform type for Laguerre and Hermite orthogonal expansions. To explain our results, consider the system of Laguerre functions, for fixed $\alpha>-1$, given by
\begin{equation*}
 l_k^\alpha(x)= \left( \frac{k!}{\Gamma(k+\alpha+1)}
\right)^{\frac12} e^{-\frac{x}{2}} L_k^\alpha(x) \ ,\quad k \in \N_0 
\end{equation*}
where $L_k^\alpha(x)$ are the Laguerre polynomials. 
The $l_k^\alpha(x)$ are eigenfunctions with eigenvalues $\lambda_{\alpha,k}= k + (\alpha+1)/2$ of the differential operator
\begin{equation}
\label{laguerre}
L= -\left( x \frac{d^2}{dx^2} + (\alpha+1) \frac{d}{dx} - \frac{x}{4} \right)
\end{equation}
and are an orthonormal basis in $L^2(\R_{+},x^\alpha dx)$. Therefore, for $\gamma<p(\alpha+1)-1$ we can associate
to any $f \in L^p(\mathbb{R}_+, x^\gamma \, dx)$  its Laguerre series:
\begin{equation*} f(x) \sim \sum_{k=0}^\infty  a_{\alpha,k}(f) l_k^\alpha(x), \quad   a_{\alpha,k}(f)= \int_0^\infty f(x) l_k^\alpha(x) x^\alpha dx \label{Laguerre-series}
\end{equation*}
and, given a bounded sequence $\{m_k\}$, we can define a multiplier operator by
\begin{equation}
 M_{\alpha,m} f(x) \sim \sum_{k=0}^\infty  a_{\alpha,k}(f) m_k l_k^\alpha(x).
\label{multiplier-operator}
 \end{equation}

The main example of the kind of multipliers we are interested in is the Laguerre fractional integral,  introduced
by G. Gasper, K. Stempak and W. Trebels in \cite{GST} as an analogue  in the Laguerre setting of the classical fractional
integral of Fourier analysis, and  given by
\begin{equation*}
I_\sigma f(x) \sim \sum_{k=0}^\infty (k+1)^{-\sigma} a_{\alpha,k} l_k^\alpha(x).
\end{equation*}

In \cite{GST} the aforementioned authors obtained weighted estimates for this operator that were later improved by G. Gasper and W. Trebels in \cite{GT} using a completely different proof. In this work we recover some of the ideas of the original method of \cite{GST}, but simplifying the proof in many technical details and extending it to obtain a better range of exponents that, in particular, give the same result of \cite{GT} for the Laguerre fractional integral. Moreover, we show that our proof applies to a wide class of multipliers, namely multipliers arising from a Laplace-Stieltjes transform, which are of the form \eqref{multiplier-operator} with  $m_k=m(k)$ given by the Laplace-Stieljtes transform  of some real-valued function 
$\Psi(t)$, that is,
\begin{equation}
m(s) = \mathfrak{L}\Psi(s) := \int_0^\infty e^{-st} d\Psi(t).
 \label{Laplace.transform} 
 \end{equation}

We will assume that $\Psi$ is  of bounded variation in $\R_+$,  so that the Laplace transform converges absolutely in the half plane $\hbox{Re}(s)\geq 0$  (see \cite[Chapter 2]{Widder}) and the definition of the operator $M_{\alpha,m}$ makes sense.

Multipliers of this kind are quite natural to consider and, indeed, a slightly different definition is given by E. M. Stein in \cite{Stein} and was previously used in the unweighted setting by E. Sasso in \cite{S}. More recently, B. Wr\'obel \cite{W} has obtained weighted $L^p$ estimates for the both the kind of multipliers considered in  \cite{Stein} and the ones considered here when $\alpha \in \{-\frac12\} \cup [\frac12,\infty)$, by proving that they are Calder\'on-Zygmund operators  (see Section 4 below for a precise comparison of results). Also, let us mention that T. Mart\'inez has considered multipliers of Laplace transform type for ultraspherical expansions  in \cite{Martinez}. 

 Other kind of multipliers for Laguerre expansions have also been considered, see, for instance, \cite{GST, Stempak-Trebels, Thangavelu} where boundedness criteria are given in terms of difference operators. In our case, we will only require minimal assumptions on the function $\Psi$,   which are more natural in our context, and easier to verify in the case of the Laguerre fractional integral and in other examples that we will consider later. Indeed, the main theorem we will prove for multipliers for Laguerre expansions reads as follows:
 \begin{theorem}
Assume that  $\alpha>-1$ and that $M_{\alpha,m}$ is a multiplier of Laplace transform type for Laguerre expansions, given by \eqref{multiplier-operator}  and  \eqref{Laplace.transform}, such that:
\begin{enumerate}
\item[(H1)] \begin{equation*}\int_0^\infty |d\Psi|(t)< +\infty; \end{equation*}
\item[(H2)] there exist $\delta>0$,  $0 < \sigma < \alpha+1$, and $C>0$ such that
$$ |\Psi(t)| \leq C t^{\sigma} \quad \hbox{for} \; 0 \le t \leq \delta .$$
\end{enumerate}

Then $M_{\alpha,m}$ can be extended to a bounded operator such that
$$ \| M_{\alpha,m} f \|_{L^q(\mathbb{R}_+, x^{(\alpha-bq)})} \leq C \| f \|_{L^p(\mathbb{R}_+, x^{(\alpha+ap)})} $$
provided that the following conditions hold:
\begin{equation*}
1 < p \leq q < \infty \quad , \quad a < \frac{\alpha+1}{p^\prime}\quad , \quad  b < \frac{\alpha+1}{q}
\end{equation*} 
and 
\begin{equation*}
\label{cond19}
 \left( \frac{1}{q} - \frac{1}{p} \right) \left(\alpha+\frac12\right) \le a+b  \le \left(\frac{1}{q}-\frac{1}{p}\right)(\alpha+1) + \sigma.
\end{equation*} 
\label{main-result}
\end{theorem}

Besides the system $\{l_k^\alpha\}_{k\ge 0}$, other families of Laguerre functions have been considered in the literature, and using an idea due to I. Abu-Falah, R. A. Mac\'ias, C. Segovia and J. L. Torrea \cite{AMST} we will show that analogues of Theorem \ref{main-result} hold for those families with appropriate changes in the exponents (see Section 3 for the precise statement of results). 

Finally, the well-known connection between Laguerre and Hermite expansions will allow us to extend the above result to an analogous result for Laplace type multipliers for Hermite expansions. To make this precise, recall that, given $f \in L^2(\mathbb{R})$, we can consider its Hermite series expansion
\begin{equation*}
f \sim \sum_{k=0}^\infty c_k(f) h_k ,  \quad   c_{k}(f)= \int_{-\infty}^\infty f(x) h_k(x) dx
 \label{Hermite-series}. 
\end{equation*}
where $h_k$ are the Hermite functions given by 
\begin{equation*}
h_k(x)= \frac{(-1)^k}{(2^k k! \pi^{1/2})^{1/2}} H_k(x) e^{-\frac{x^2}{2}}, 
\end{equation*}
which are the normalized eigenfunctions of the Harmonic oscillator operator 
$$H=-\frac{d^2}{dx^2} +  |x|^2.
$$

As before, given a bounded sequence $\{m_k\}$ we can define a multiplier operator by
\begin{equation}
\label{hermite-multiplier}
M_{H,m} f \sim \sum_{k=0}^\infty c_k(f) m_k h_k
\end{equation}
and we say that it is a Laplace transform type multiplier if equation \eqref{Laplace.transform} holds. Then, we have the following analogue of Theorem \ref{main-result}, which, in the case of the Hermite fractional integral (that is, for $m_k= (2k+1)^{-\sigma}$), gives the same result of \cite[Theorem 2.5]{Nowak-Stempak} in the one-dimensional case:

\begin{theorem}
\label{teorema-hermite}
Assume that $M_{H,m}$ is a multiplier of Laplace transform type for Hermite expansions, given by \eqref{hermite-multiplier}  and  \eqref{Laplace.transform}, such that:
\begin{enumerate}
\item[(H1h)] $$ \int_0^\infty |d\Psi|(t) < +\infty;$$
\item[(H2h)] there exist $\delta>0$,  $0 < \sigma < \frac12$,  and $C>0$ such that
$$ |\Psi(t)| \leq C t^{\sigma} \quad \hbox{for} \; 0 \le t \leq \delta.$$
\end{enumerate}

Then $M_{H,m}$ can be extended to a bounded operator such that
$$ \| M_{H,m} f \|_{L^q(\mathbb{R}, x^{-bq})} \leq C \| f \|_{L^p(\mathbb{R}, x^{ap})} $$
provided that the following conditions hold:
\begin{equation*}
1 < p \leq q < \infty \quad , \quad a<\frac{1}{p'} \quad , \quad b<\frac{1}{q}
\end{equation*}
and
\begin{equation*}
0\le a+b \le \frac{1}{q} -\frac{1}{p}+ 2\sigma.
\label{escalah}
\end{equation*}

\end{theorem}

The remainder of this paper is organized as follows.  In Section 2 we prove Theorem \ref{main-result}. For the case $\alpha \ge 0$ the proof relies on the representation of the operator as a twisted generalized convolution, already used in \cite{GST} for the Laguerre fractional integral. However, instead of using the method of that paper to obtain weighted bounds, we give a simpler proof based on the use of Young's inequality in the multiplicative group $(\mathbb{R}_+, \cdot)$, which allows us to obtain a wider range of exponents. Moreover, we obtain an estimate for the convolution kernel which simplifies and generalizes Lemma 2.1 from \cite{GST}.  For the case $-1 < \alpha < 0$ the result is obtained from the previous case by means of a weighted transplantation theorem from \cite{Garrigos}. A similar idea was used by Y. Kanjin and E. Sato in \cite{KS} to prove unweighted estimates for the Laguerre fractional integral using a transplantation theorem from \cite{K}.
In Section 3 we obtain the analogues of Theorem  \ref{main-result} for other Laguerre systems using an idea from \cite{AMST}.  In Section 4 we exploit the relation between Laguerre and Hermite expansions to derive Theorem \ref{teorema-hermite} from Theorem \ref{main-result}. 
 Finally, in Section 5 we present some examples of operators covered by  the two main theorems and make some further comments.

\section{Proof of the theorem in the Laguerre case}

In this section we prove Theorem \ref{main-result}. We will divide the proof in three steps:
\begin{enumerate}
\item We write the operator as a twisted generalized convolution and obtain  the estimate for the convolution kernel when $\alpha \ge 0$. This part of the  proof follows essentially the ideas of \cite{GST}, but in the more general setting of multipliers of Laplace transform type. In particular, we provide an easier proof of the analogue of \cite[Lemma  2.1]{GST} in this setting (see Lemma \ref{lemma-g} below).
\item We complete the proof of the theorem in the case $\alpha \ge 0$ by proving weighted estimates for the generalized euclidean convolution.
\item We extend the results to the case $-1<\alpha < 0$ using the case $\alpha \ge 0$ and a weighted transplantation theorem from \cite{Garrigos} (Lemma \ref{lema-garrigos} below).
\end{enumerate}

\subsection{Step 1: representing the multiplier operator as a twisted  generalized convolution when $\alpha \ge 0$}

Following \cite{Mc,A} we define the twisted generalized convolution of $F$ and $G$ by
$$ F \times G := \int_0^\infty \tau_x F(y) \, G(y) \, y^{2\alpha+1} \, dy$$
where the twisted translation
operator is defined by
$$ \tau_x F(y)= \frac{\Gamma(\alpha+1)}{\pi^{1/2} \Gamma(\alpha+1/2)}
\int_0^\pi F((x,y)_\theta) \mathcal{J}_{\alpha-1/2}(xy \sin \theta)
(\sin \theta )^{2\alpha} \; d\theta $$
with
$$\mathcal{J}_\beta(x)= \Gamma(\beta+1) J_\beta(x)/(x/2)^\beta $$
 $J_\beta(x)$ being the Bessel function of order $\beta$ and
$$ (x,y)_\theta= (x^2 + y^2 - 2xy \cos \theta)^{1/2}.$$
Then,  we have (formally) that
\begin{equation}
\label{malfa}
M_{\alpha,m} f(x^2)= F \times G
\end{equation}
where
$$ F(y)=f(y^2)\quad , \quad G(y) = g(y^2) $$
and
\begin{equation} g(x) \sim \frac1{\Gamma(\alpha+1)}
\sum_{k=0}^\infty m_k L_k^\alpha(x) e^{-\frac{x}{2}}.
\label{series-g}
\end{equation}
Recalling that $|\mathcal{J}_\beta(x)| \leq C_\beta$  if
$\beta \geq -\frac12$, 
we have that:
\begin{equation}
 | F \times G | \le C (|F| \star |G|)  \label{convolution-bound} 
 \end{equation}
where $\star$ denotes the generalized Euclidean convolution which is defined by
\begin{equation}
\label{gen-eucl}
F \star G(x) := \int_0^\infty \tau^E_x F(y) \, G(y)  \, y^{2\alpha+1} \, dy
\end{equation}
with
\begin{equation}
\label{gen-trans}
 \tau^E_x F(y):= \frac{\Gamma(\alpha+1)}{\pi^{1/2} \Gamma(\alpha+1/2)}
\int_0^\pi F((x,y)_\theta) (\sin \theta )^{2\alpha} \; d\theta. 
\end{equation}

As a consequence of \eqref{malfa} and \eqref{convolution-bound}, the operator $M_{\alpha,m}$ is pointwise bounded by a generalized euclidean convolution with the kernel $G$ (with respect to the measure $x^{2\alpha+1} \, dx$). Therefore, we need to obtain an appropriate estimate for $G(x)=g(x^2)$, that essentially is:
$$ |g(x)| \leq C x^{\sigma-\alpha-1} \; \hbox{for} \; \alpha \geq 0 \; \hbox{and} \; 0 < \sigma < \alpha +1 $$
(see Lemma \ref{lemma-g} below for a precise statement).
 
This generalizes the result given in \cite[Lemma  2.1]{GST} but, while in that paper the proof of the corresponding estimate is based on delicate pointwise estimates for the Laguerre functions, our proof is  based on the following generating function for the Laguerre polynomials (see, for instance, \cite{Thangavelu}):
\begin{equation} 
\label{generating-function}
\sum_{k=0}^\infty L_k^\alpha(x) w^k = (1-w)^{-\alpha-1} e^{-\frac{xw}{1-w}} :=  Z_{\alpha,x}(w)  \quad (|w|<1).
\end{equation}

To explain our ideas, we point out that if the series in \eqref{series-g} were convergent (this need not be the case) we would have:
\begin{align*} g(x) &= \frac1{\Gamma(\alpha+1)}
\sum_{k=0}^\infty m_k L_k^\alpha(x) e^{-\frac{x}{2}}
\\ & = \frac1{\Gamma(\alpha+1)}
\sum_{k=0}^\infty \left( \int_0^\infty e^{-kt} d\Psi(t) \right)  L_k^\alpha(x) e^{-\frac{x}{2}}
\\ & =  \frac1{\Gamma(\alpha+1)} e^{-\frac{x}{2}}
\int_0^\infty  Z_{\alpha,x}(e^{-t})  \; d\Psi(t).
\end{align*}

The main advantage of this formula is that it shields a rather explicit expression for $g$ in which, thanks to \eqref{generating-function}, the Laguerre polynomials do not appear.

However, in general it is not clear if the series in  \eqref{series-g} is convergent (not even in the special
case of the Laguerre fractional integral $m(t)=t^{\sigma-1}$). Moreover, the integration of the series in $Z_{\alpha,x}(w)$
is difficult to justify since it is not uniformly convergent in the interval $[0,1]$ (because $Z_{\alpha,x}(w)$ is not analytical for $w=1$).

Nevertheless, we will see that the formal manipulations above can be given a rigorous meaning if we agree in understanding the convergence  of the series in $\eqref{series-g}$ in the Abel sense. For this purpose, we introduce a regularization parameter $\rho \in (0,1)$, we consider the regularized function
\begin{equation}
 g_{\rho}(x) = \frac1{\Gamma(\alpha+1)} \sum_{k=0}^\infty m_k \rho^k  L_k^\alpha(x) e^{-\frac{x}{2}}
\label{series-g-rho} 
\end{equation}
and recall that the series in \eqref{series-g} is summable in Abel sense to the limit $g(x)$ if there exists the limit
\begin{equation*}
 g(x) = \lim_{\rho \to 1} g_{\rho}(x). 
  \end{equation*}
 
 With this definition in mind, we can give a rigorous meaning to the heuristic idea described above. More precisely, we will prove the following:

\begin{lemma} \label{lemma-g}
Let $ g_{\rho}$ be defined by \eqref{series-g-rho}. Then:

(1) For $0<\rho<1$ the series \eqref{series-g-rho} converges absolutely.

(2) The following representation formula holds:
\begin{equation}
\label{rep-grho}
g_{\rho}(x) =  \frac{1}{\Gamma(\alpha+1)} \int_0^\infty Z_{\alpha,x}(\rho e^{-t}) \; d\Psi(t).
\end{equation}

(3) If we define $g(x)$ by setting $\rho=1$ in this representation formula,
$g(x)$ is well defined and the series \eqref{series-g} converges to $g(x)$ in the Abel sense.

(4) If $\alpha>0$, $0 <\rho_0  < \rho  \leq 1$  and $0 < \sigma < \alpha +1$, then
$$ |g_\rho (x)| \leq C x^{\sigma-\alpha-1}, $$
with a constant $C=C(\alpha,\sigma)$ independent of $\rho$.
\end{lemma}

\begin{proof} 

(1)  Observe first  that hypothesis $(H1)$ implies that $(m_k)$ is a bounded sequence. Indeed, 
$$ 
|m_k|  \leq \int_0^\infty e^{-kt} |d\phi|(t) \leq \int_0^\infty |d\phi|(t) = C < +\infty.
$$

Now recall that (\cite[Lemma 1.5.3]{Thangavelu}), if $\nu=\nu(k)= 4k + 2\alpha +2$, 
$$
 |l_k^\alpha(x)| \leq C (x\nu)^{-\frac14} \quad \hbox{if } \frac{1}{\nu} \leq
x \leq \frac{\nu}{2}.
$$
 Therefore, if we fix $x$, for $k \geq k_0$, $x$ is in the region where
this estimate holds  (since $\nu \to +\infty$ when $k \to +\infty$), and from Stirling's formula we deduce that
$$
 \frac{k!}{\Gamma(k+\alpha+1)} =  \frac{\Gamma(k+1)}{\Gamma(k+\alpha+1)} = O(k^{-\alpha}). 
 $$
 
Then we have the following estimate for the terms of the series in \eqref{series-g-rho}
$$
 |m_k \rho^k L^\alpha_k(x)| e^{-\frac{x}{2}} \leq C(x)
\rho^k k^{-\sigma} \; \hbox{for} \; k \geq k_0, 
$$
and, since $\rho<1$, this implies that the series converges absolutely.\footnote{K. Stempak has observed that this result can be also justified by observing
that, for fixed $x$, $L^\alpha_k(x)$ has at most polynomial growth with $k\to
\infty$ (see, for instance, (7.6.9) and (7.6.10) in \cite{Sz}). Hence, the polynomial
growth of $L^\alpha_k(x)$ versus the exponential decay of $\rho^k$, with $m_k$
disregarded as a bounded sequence, produce an absolutely convergent series.}

\medskip

(2)  First, observe that $Z_{\alpha,x}(w)$ is  continuous as a function of a real variable for $w \in [0,1]$ (if we define
 $Z_{\alpha,x}(1)=0$) and, therefore, it is bounded, say
\begin{equation*} 
|Z_{\alpha,x}(w)| \leq C = C(\alpha,x) \; \hbox{for} \; w \in [0,1]. \label{Z-bound}
 \end{equation*}
 
Hence, using hypothesis $(H1)$ we see that the integral in the representation formula is convergent for any $\rho \in [0,1]$.  Moreover, from our assumptions we have that, for $\rho<1$,
\begin{align}
\nonumber g_{\rho}(x) & =  \frac{1}{\Gamma(\alpha+1)}
\sum_{k=0}^\infty m_k \rho^k L_k^\alpha(x) e^{-\frac{x}{2}}
\\ \nonumber & = \frac{1}{\Gamma(\alpha+1)}
\sum_{k=0}^\infty \left( \int_0^\infty \rho^k e^{-kt} d\Psi(t) \right)  L_k^\alpha(x) e^{-\frac{x}{2}}
\\ \nonumber & = \lim_{N \to +\infty} \frac{1}{\Gamma(\alpha+1)}
\sum_{k=0}^N \left( \int_0^\infty \rho^k e^{-kt} d\Psi(t) \right)  L_k^\alpha(x) e^{-\frac{x}{2}}
\\ & = \lim_{N \to +\infty} \frac{1}{\Gamma(\alpha+1)} e^{-\frac{x}{2}} \int_0^\infty Z_{\alpha,x}^{(N)}(\rho e^{-t}) \; d\Psi(t) \label{limN}
\end{align}
where
$$ 
Z_{\alpha,x}^{(N)}(w) = \sum_{k=0}^N L_k^\alpha(x) w^k
$$
denotes a partial sum of the series for $Z_{\alpha,x}(w)$. Now, since $\rho<1$, that series converges uniformly in the
interval $[0,\rho]$, so that given $\varepsilon>0$ there exists $N_0=N_0(\varepsilon)$ such that
$$
 |Z_{\alpha,x}(w)-Z_{\alpha,x}^{(N)}(w)| < \varepsilon \; \hbox{if} \; N \geq N_0.
 $$
 
Using this estimate and hypothesis $(H1)$, we obtain
\begin{align*}
& \left|  \int_0^\infty Z_{\alpha,x}(\rho e^{-t}) \; d\Psi(t)
-  \int_0^\infty Z_{\alpha,x}^{(N)}(\rho e^{-t}) \; d\Psi(t) \right| 
\\ & \leq \int_0^\infty
|Z_{\alpha,x}(\rho e^{-t})-Z_{\alpha,x}^{(N)}(\rho e^{-t})| \; |d\Psi|(t) 
\\ &\leq C \varepsilon  
\end{align*}
 from which we conclude that
\begin{equation}
\label{limZN}
\lim_{N \to +\infty}  \int_0^\infty Z_{\alpha,x}^{(N)}(\rho e^{-t}) \; d\Psi(t)
=   \int_0^\infty Z_{\alpha,x}(\rho e^{-t}) \; d\Psi(t) 
\end{equation}
and, replacing \eqref{limZN} into  \eqref{limN} we obtain \eqref{rep-grho}.

\medskip

(3) We have already observed that the integral in \eqref{rep-grho} is convergent for $\rho=1$. Moreover, the
bound we have proved above for $Z_{\alpha,x}$, and $(H1)$ imply that we can apply the Lebesgue bounded convergence theorem to this integral (with a constant majorant function, which is integrable with respect to $|d\Psi|(t)$ by $(H1)$), to conclude that $g(x)=\lim_{\rho \to 1}g_\rho(x)$.

\medskip

(4) Let $\delta$ be as in $(H2)$ and observe that 
\begin{align*}
\Gamma(\alpha+1) g_\rho(x) &=  e^{-\frac{x}{2}} \int_0^\infty Z_{\alpha,x}(\rho e^{-t}) d\Psi(t)
\\ &=   e^{-\frac{x}{2}} \int_0^\delta Z_{\alpha,x}(\rho e^{-t}) d\Psi(t) + e^{-\frac{x}{2}}  \int_\delta^\infty Z_{\alpha,x}(\rho e^{-t}) d\Psi(t) 
\\ & =  \underbrace{   e^{-\frac{x}{2}} \int_0^\delta Z_{\alpha,x}^\prime (\rho e^{-t}) \rho e^{-t} \Psi(t) \, dt}_{(i)} + \underbrace{ e^{-\frac{x}{2}} Z_{\alpha,x}(\rho e^{-\delta}) \Psi(\delta)}_{(ii)}
\\ & \quad - \underbrace{e^{-\frac{x}{2}}   Z_{\alpha,x}(\rho) \Psi(0)}_{(iii)}  +  \underbrace{  e^{-\frac{x}{2}}  \int_\delta^\infty Z_{\alpha,x}(\rho e^{-t}) d\Psi(t) }_{(iv)}
\end{align*}

Since  $|Z_{\alpha,x}(\rho e^{-\delta})|\le (1-\rho e^{-\delta})^{-\alpha-1} \le C_\delta$,  $\Psi(0)=0$,  and  $\sigma-\alpha-1<0$,  clearly 
$(ii) \le C x^{\sigma-\alpha-1}$ and $(iii)$ vanishes.

To bound $(iv)$, notice that if $\omega= \rho e^{-t}$ and $t>\delta$, $0\le Z_{\alpha,x}(\omega) \le M_\delta$. Therefore, using $(H1)$ and the fact that $\sigma-\alpha-1<0$ we obtain
\begin{equation*}
(iv)  \le e^{-\frac{x}{2}} M_\delta \int_\delta^\infty |d \Psi|(t)  \le C x^{\sigma-\alpha-1}.
\end{equation*}

Now, observing that 
\begin{equation*}
Z_{\alpha,x}^{\prime}(\omega) = (\alpha+1) Z_{\alpha+1,x}(\omega) - x Z_{\alpha+2,x}(\omega).
\end{equation*}
and using $(H2)$, we obtain
\begin{align*}
(i) & \le C e^{-\frac{x}{2}} \int_0^\delta Z_{\alpha+1,x}(\rho e^{-t}) \rho e^{-t} t^{\sigma} \, dt 
\\ & \quad + e^{-\frac{x}{2}} \int_0^\delta x Z_{\alpha+2,x}(\rho e^{-t}) \rho e^{-t} t^{\sigma} \, dt
\end{align*}
and the wanted estimates in this case follow by  a direct application of the following lemma.
\end{proof}

\begin{lemma}
In the conditions of Lemma \ref{lemma-g}(4), if
$$I(x) = e^{-\frac{x}{2}} \int_0^\delta Z_{\beta,x}(\rho e^{-t}) \rho e^{-t} t^{\sigma} \, dt,$$
and $\beta=\alpha+1$ or $\beta=\alpha+2$ then, $|I(x)| \le C x^{\sigma-\beta}$ with $C=C(\beta, \sigma, \delta, \rho_0)$.
\end{lemma}

\begin{proof}
Making the change of variables $w= \rho e^{-t}$, and recalling the definition of $Z_{\beta,x}(w)$ given by \eqref{generating-function}, 
we see that
\begin{align*} 
I(x) & 
 =  e^{-\frac{x}{2}} \int_{\rho e^{-\delta}}^\rho (1-w)^{-\beta-1} e^{-\frac{xw}{1-w}}  \log^{\sigma}\left(\frac{\rho}{w}\right) \, dw 
\end{align*}

Making a further change of variables $u=\frac12 + \frac{w}{1-w}$ and setting $c_\delta = e^{-\delta}$ this is
\begin{align}
\nonumber I(x) & = \int_{\frac12 + \frac{c_\delta \rho}{1-c_\delta \rho}}^{\frac12 + \frac{\rho}{1-\rho}} \left(u+\frac12\right)^{\beta+1} e^{-ux} \left[ \log\left( \rho \frac{u+\frac12}{u-\frac12} \right)\right]^{\sigma} \frac{1}{\left( u+\frac12 \right)^2} \, du
\\ & \le C \int_{\frac12 + \frac{c_\delta \rho}{1-c_\delta \rho}}^{\frac12 + \frac{\rho}{1-\rho}} u^{\beta-1} e^{-ux} \left( u - \frac12 \right)^{-\sigma} \underbrace{\left[u(\rho-1)+\frac12 (\rho+1)\right]^{\sigma}}_{:= \tilde u(\rho)} \, du \label{just}
\end{align}
where in \eqref{just} we have used that,  since 
$$ \rho \frac{u+\frac12}{u-\frac12}  = 1+\frac{u(\rho-1) + \frac12 (\rho+1)}{u-\frac12},$$
then 
$$\log\left( \rho \frac{u+\frac12}{u-\frac12} \right) \le \frac{u(\rho-1) + \frac12 (\rho+1)}{u-\frac12}.$$

Since  $\frac12 <  u\le \frac12 + \frac{\rho}{1-\rho}$,  it is immediate that 
$$0 \le u(\rho-1)+ \frac12(\rho+1)\le \rho,$$
 which, using that $\sigma \ge 0$,  implies  $ \tilde u(\rho)\le 1.$

 Also, since 
 $$
u \ge \frac12 + \frac{c_\delta \rho_0}{1- c_\delta \rho_0} > \frac12 $$
 we have that 
 $$
 \left(u-\frac12 \right)^{-\sigma} \le C u^{-\sigma}
 $$
 where the constant depends only on $\rho_0$ and $\delta$.
  Therefore, 
\begin{align}
\nonumber I(x) & \le C \int_0^\infty u^{\beta-\sigma-1} e^{-ux} \, du
\\ & =   C x^{-\beta+\sigma} \int_0^\infty v^{\beta-\sigma-1} e^{-v} \, dv \label{r1}
\\ & \le C x^{-\beta+\sigma} \label{r2}
\end{align}
where in \eqref{r1} we have made the change of variables $v=ux$, and in \eqref{r2} we have used that  $\beta-\sigma-1> -1$ because $\beta=\alpha+1$ or $\beta=\alpha+2$.

\end{proof}

\subsection{Step 2: weighted estimates for the generalized Euclidean convolution}

Following the idea of the previous section, we define a regularized multiplier operator
$M_{\alpha,m,\rho}$ by:
\begin{equation}
\label{Mamr}
 M_{\alpha,m,\rho} f(x):= \sum_{k=0}^\infty m_k \rho^k a_{k,\alpha}(f) l_
k^\alpha(x) 
\end{equation}

In this section we will obtain the estimate
\begin{equation}
\label{acotacion}
 \left( \int_0^\infty  |M_{\alpha,m,\rho}(f)|^q x^{\alpha-bq} \; dx \right)^{\frac{1}{q}} \leq
C \left( \int_0^\infty |f|^p x^{\alpha+ap} \; dx \right)^{\frac{1}{p}} 
\end{equation}
for  $f \in L^p(\mathbb{R}_+, x^{\alpha +ap})$ with a constant $C$ independent of the regularization parameter $\rho$ and appropriate $a,b$ (see Theorem \ref{teo-convolucion}). 

Indeed, the operator can be expressed as before as a twisted generalized convolution
with kernel $G_{\rho}(y)=g_\rho(y^2)$ (in place of $G$), and by Lemma \ref{lemma-g}, if $F(y)=f(y^2)$, we have the pointwise bound
$$ |M_{\alpha,m,\rho} f(x^2)| \leq (|F| \star |G_\rho|)(x) \leq C ( |F| \star |x^{2(\sigma-\alpha-1)}|)(x).$$
Therefore, \eqref{acotacion} will follow from a weighted inequality for the generalized Euclidean convolution
with kernel $K_\sigma := x^{2(\sigma-\alpha-1)}$ (Theorem \ref{teo-convolucion}).

Once we have \eqref{acotacion}, Theorem \ref{main-result} will follow by a standard density argument. Indeed, if we consider the space
$$ E= \{ f(x)=p(x) e^{-\frac{x}{2}} : 0 \leq  x, \, p(x) \mbox{ a polynomial} \}, $$
any $f \in E$ has only  a finite number of non-vanishing Laguerre coefficients. In that case, it is straightforward that
$ M_{\alpha,m} f(x)$ is well-defined and:
$$
M_{\alpha,m} f(x) = \lim_{\rho \to 1} M_{\alpha,m,\rho} f(x).
$$

Then, by Fatou's lemma, 
$$
 \int_0^\infty |M_{\alpha,m}(f)|^q x^{\alpha-bq} \; dx
\le \lim_{\rho \to 1} \int_0^\infty   |M_{\alpha,m,\rho}(f)|^q x^{\alpha-bq} \; dx
$$
and, therefore, we obtain
$$ \left( \int_0^\infty  |M_{\alpha,m,\rho}(f)|^q x^{\alpha-bq} \; dx \right)^{\frac{1}{q}} \leq
 C \left( \int_0^\infty |f|^p x^{\alpha+ap} \; dx \right)^{\frac{1}{p}}  \;
\forall f \in E.$$

Since $E$  is dense in $L^p(\mathbb{R}_+,x^{\alpha+a p})$,
we deduce that $M_{\alpha,m}$ can be extended to a bounded operator from $L^p(\mathbb{R}_+, x^{\alpha +ap})$ to $L^q(\mathbb{R}_+, x^{\alpha -bq})$. Moreover, the extended operator satisfies:

$$ M_{\alpha,m} f = \lim_{\rho \to 1} M_{\alpha,m,\rho} f.$$

This means that the formula \eqref{multiplier-operator} is valid for $f\in L^p(\mathbb{R}_+, x^{\alpha +ap})$ if the summation is interpreted in the Abel sense with convergence in $L^q(\mathbb{R}_+, x^{\alpha -bq})$. Therefore, to conclude the proof of Theorem \ref{main-result} in the case $\alpha\ge 0$ it is enough to see that the following result holds:
 
\begin{theorem}
\label{teo-convolucion}
Let $\alpha \ge 0,  0<\sigma<\alpha+1$ and $M_{\alpha,m,\rho}$ be given by \eqref{Mamr} such that it satisfies $(H1)$ and $(H2)$. Then, for all $f\in L^p(\mathbb{R}_+, x^{\alpha+ap})$, the following estimate holds
\begin{equation*}
 \|M_{\alpha,m,\rho}f(x^2) x^{-2b} \|_{L^q(\mathbb{R}_+, x^{2\alpha+1})}
\leq  \| f(x^2) x^{2a} \|_{L^p(\mathbb{R}_+, x^{2\alpha+1})}
\end{equation*}
 provided that 
\begin{equation*}
a<\frac{\alpha+1}{p'} \quad , \quad b<\frac{\alpha+1}{q}
\end{equation*}
and that
\begin{equation*}
  \left(\frac{1}{q}-\frac{1}{p} \right)\left(\alpha +\frac12 \right) \le a+b \le \left( \frac{1}{q} -\frac{1}{p}\right) (\alpha+1) + \sigma
\end{equation*}
\end{theorem}

\begin{proof}
First, notice that if condition $(H2)$ holds for a certain $0<\sigma_0<\alpha+1$, then it also holds for any $0<\sigma < \sigma_0$. Therefore,  it suffices to prove the theorem in the case $a+b = (\frac{1}{q} - \frac{1}{p})(\alpha+1) + \sigma$ which in turn, by the conditions above, implies $\sigma \ge -\frac12 \left( \frac{1}{q}-\frac{1}{p}\right)$. 

Let $K_\sigma (x) := x^{2(\sigma-\alpha-1)}$, $F(y)=f(y^2)$  and recall that 
$$|M_{\alpha,m,\rho}f(x^2)| \le C (|F|\star |K_\sigma|)(x)$$
 where $\star$ denotes the generalized euclidean convolution defined by \eqref{gen-eucl}.

We begin by computing the generalized Euclidean translation of $K_\sigma$ given by \eqref{gen-trans}. Making the change of variables
$$
t=\cos \theta \Rightarrow dt = - \sin \theta \, d\theta
= - \sqrt{1-t^2} \, d\theta
$$ 
we see that
$$ \tau_x^E K_\sigma(y)= C(\alpha)
\int_{-1}^1 (x^2+y^2-2xyt)^{\sigma-\alpha-1} (1-t^2)^{\alpha-\frac12} \; dt .$$
Following the notation of our previous work \cite{ddd}, if we let 
$$ I_{\gamma,k}(r):= \int_{-1}^1 \frac{(1-t^2)^k}{(1-2rt+r^2)^{\frac{\gamma}{2}}}\; dt, $$
then
$$ \tau_x^E K_\sigma(y)=  C(\alpha)
y^{2(\sigma-\alpha-1)}  I_{2(1+\alpha-\sigma), \alpha-\frac12}\left( \frac{x}{y} \right)$$
and, therefore,
\begin{align}
\nonumber K_\sigma \star F(x) & = C \int_0^\infty y^{2(\sigma-\alpha-1)}  I_{2(1+\alpha-\sigma),\alpha-\frac12}\left( \frac{x}{y} \right) F(y) y^{2\alpha+1} dy  
\\ & = C \int_0^\infty y^{2\sigma}
I_{2(1+\alpha-\sigma),\alpha-\frac12}\left( \frac{x}{y} \right) F(y) \frac{dy}{y} \label{star}
\end{align}

Now,
\begin{align*}
 \|M_{\alpha,m,\rho}f(x^2) x^{-2b} \|_{L^q(\mathbb{R}_+, x^{2\alpha+1})} & \le C \| [ K_\sigma  \star F(x)] x^{-2b} \|_{L^q(\mathbb{R}_+,x^{2\alpha+1})} 
\\ & =  C \left( \int_0^\infty |  K_\sigma  \star F(x) x^{-2b} |^q x^{2\alpha+1} \; dx  \right)^{\frac{1}{q}}
 \\ & =  C \left( \int_0^\infty \left| K_\sigma  \star F(x)  x^{\frac{2\alpha+2}{q}-2b}
\right|^q\; \frac{dx}{x} \right)^{\frac{1}{q}}
\end{align*}
but, by \eqref{star},
\begin{align*}
 [  K_\sigma & \star F(x) ]  x^{\frac{2\alpha+2}{q}-2b}
\\ & =  C \int_0^\infty y^{2\sigma} x^{\frac{2\alpha+2}{q}-2b}
I_{2(1+\alpha-\sigma),\alpha-\frac12}\left( \frac{x}{y} \right) F(y) \frac{dy}{y} 
\\ & =  C \int_0^\infty
\left(\frac{y}{x}\right)^{-[\frac{2\alpha+2}{q}-2b]}
I_{2(1-\alpha-\sigma),\alpha-\frac12}\left( \frac{x}{y} \right) F(y)
y^{2\sigma+\frac{2\alpha+2}{q}-2b} \frac{dy}{y} 
\\ & = [ y^{\frac{2\alpha+2}{q}-2b} I_{2(1+\alpha-\sigma),\alpha-\frac12}(y) *
 F(y) y^{2\sigma+\frac{2\alpha+2}{q}-2b} ](x)
 \end{align*}
where $*$ denotes the convolution in $\mathbb{R}_+$ with respect to the Haar measure $\frac{dx}{x}$.

Then, by Young's inequality:
\begin{align*}
\|M_{\alpha,m,\rho} & f(x^2) x^{-2b} \|_{L^q(\mathbb{R}_+, x^{2\alpha+1})} 
\\ & \leq \| F(x) x^{2\sigma+\frac{2\alpha+2}{q}-2b}
\|_{L^p\left(\frac{dx}{x} \right)}
\| x^{\frac{2\alpha+2}{q}-2b} I_{2(1+\alpha-\sigma),\alpha-\frac12}(x)
\|_{L^{s,\infty}(\frac{dx}{x})} 
\end{align*}

provided that:
\begin{equation}
\label{youngl}
 \frac{1}{p}+\frac{1}{s}=1+\frac{1}{q}.
\end{equation}

Since we are assuming that $a+b=\left( \frac{1}{q}-\frac{1}{p}\right)(\alpha +1)+\sigma$, we have that
\begin{align*}
\| F(x) x^{2\sigma+\frac{2\alpha+2}{q}-2b} \|_{L^p\left(\frac{dx}{x} \right)}  & = \left( \int_0^\infty | F(x) x^{2\sigma+\frac{2\alpha+2}{q}-2b} |^p  \; \frac{dx}{x}\right)^{\frac{1}{p}}
\\ & = \left( \int_0^\infty | F(x) x^{2a+\frac{2\alpha+2}{p}} |^p  \; \frac{dx}{x}\right)^{\frac{1}{p}}
\\ & = \| F(x) x^{2a} \|_{L^p(\mathbb{R}_+,x^{2\alpha+1})}
\\ & = \| f(x^2) x^{2a} \|_{L^p(\mathbb{R}_+,x^{2\alpha+1})}
\end{align*}
whence, to conclude the proof of the theorem it suffices to see that
$$\| x^{\frac{2\alpha+2}{q}-2b} I_{2(1+\alpha-\sigma),\alpha-\frac12}(x)
\|_{L^{s,\infty}(\frac{dx}{x})} < +\infty. $$

For this purpose, we shall use the following lemma, which is a generalization of our previous result \cite[Lemma 4.2]{ddd}. The first part of the proof is the same as in that lemma, but it is included here for the sake of completeness:

\begin{lemma}
Let
$$ I_{\gamma,k}(r)= \int_{-1}^1 \frac{(1-t^2)^k}{(1-2rt+r^2)^{\frac{\gamma}{2}}}\; dt $$

Then, for $r \sim 1$ and $k >-1$,
we have that
$$ |I_{\gamma,k}(r)| \leq
 \left\{
\begin{array}{lclcc}
C_{\gamma,k} & \mbox{if} &  \gamma<2k+2 \\
C_{\gamma,k} \log\frac{1}{|1-r|}& \mbox{if} &  \gamma=2k+2 \\
C_{\gamma,k} |1-r|^{-\gamma+2k+2} & \mbox{if} & \gamma > 2k+2\\
\end{array}
\right.
 $$

 \end{lemma}
 
\begin{proof}
Assume first that $k \in\N_0$ and $-\frac{\gamma}{2}+k > -1$. Then,
$$
I_{\gamma,k}(1)\sim\int_{-1}^1 \frac{(1-t^2)^k}{(2-2t)^{\frac{\gamma}{2}}} \, dt
\sim C \int_{-1}^1 \frac{(1-t)^k}{(1-t)^{\frac{\gamma}{2}}} \, dt.
$$
Therefore,  $I_{ \gamma,k}$ is bounded.

If $-\frac{\gamma}{2}+k =-1$, then
$$
I_{\gamma,k}(r)\sim\int_{-1}^1 (1-t^2)^k
\frac{d^k}{dt^k}\left\{(1-2rt+r^2)^{-\frac{\gamma}{2}+k}\right\}\, dt.
$$
Integrating by parts $k$ times (the boundary terms vanish),
$$
I_{\gamma,k}(r)\sim\left|\int_{-1}^1 \frac{d^k}{dt^k}\left\{(1-t^2)^k\right\}
(1-2rt+r^2)^{-\frac{\gamma}{2}+k}\, dt\right|.
$$
But $\frac{d^k}{dt^k}\left\{(1-t^2)^k\right\}$ is a polynomial of degree $k$
and therefore is bounded in $[-1,1]$ (in fact, it is up to a constant the
classical Legendre polynomial). Therefore,
$$
I_{ \gamma,k}(r) \sim \frac{1}{2r} \log\left(\frac{1+r}{1-r}\right)^2 \le C \log
\frac{1}{|1-r|}.
$$

Finally, if  $-\frac{\gamma}{2}+k <-1$, then
integrating by parts as before,
$$
I_{\gamma,k}(r)\le C_k\int_{-1}^{1}(1-2rt+r^2)^{-\frac{\gamma}{2}+k}\, dt.
$$

Thus, 
$$
I_{ \gamma,k}(r) \sim (1-2rt+r^2)^{-\frac{\gamma}{2}+k+1}|_{t=-1}^{t=1}
\le C_{k,\gamma} |1-r|^{-\gamma+2k+2}.
$$

This finishes the proof if $k\in \mathbb{N}_0$.

Consider now the case $k=m+\nu$ with $m\in \mathbb{N}_0$ and $0<\nu<1$.
Then,
\begin{align*}
I_{\gamma, k}(r) & = \int_{-1}^1 (1-t^2)^{\nu(m+1)+(1-\nu)m} (1-2rt + r^2)^{-\frac{\nu\gamma}{2}-\frac{(1-\nu)\gamma}{2}}  \, dt
\\Ê& \le I_{m+1, \gamma}^\nu(r) I_{m, \gamma}^{1-\nu}(r),
\end{align*}
where in the last line we have used  H\"older's inequality with exponent $\frac{1}{\nu}$.

If $\gamma<2m+2$, by the previous calculation
$$
|I_{\gamma,k}(r)|\le C.
$$

If $\gamma > 2(m+1)+2$, then, by the previous calculation
\begin{align*}
|I_{\gamma,k}(r)| &\le C |1-r|^{\nu(-\gamma+2(m+1)+2)}  |1-r|^{(1-\nu)(-\gamma+2m+2)} 
\\ &= C |1-r|^{-\gamma+2k+2}.
\end{align*}

For the case $2m+2 < \gamma < 2m+4 $, notice that we can always assume $r<1$, since $I_{ \gamma,k}(r) = r^{-\gamma} I_{\gamma,k}(r^{-1})$. Then, as before, we can prove that
$$
I_{\gamma,k}'(r) \le  \gamma (1-r) I_{\gamma+2,k}(r)
$$
But now we are in the case $\gamma + 2 > 2(m+1)+2$ and, thus, 
$$
|I_{\gamma+2,k}(r)| \le C |1-r|^{-\gamma+2k}.
$$

Therefore, if $-\gamma+2k+1\neq -1$
\begin{align*}
I_{\gamma,k}(r) &= \int_0^r I'_{\gamma,k}(s) \, ds
\\ & \le C \int_0^r (1-s)^{-\gamma+2k+1} \, ds 
\\ & \le C |1-r|^{-\gamma+2k+2},
\end{align*}
and if
$-\gamma+2k+1= -1$
\begin{align*}
I_{\gamma,k}(r) &\le C \int_0^r \frac{1}{1-s} \, ds 
\\ &= C \log \frac{1}{|1-r|}.
\end{align*}

It remains to check the case $k \in (-1, 0)$. For this purpose, write
$$
I_{ \gamma,k}(r) = \underbrace{\int_{-1}^0 \frac{(1-t^2)^{k}}{(1-2rt+r^2)^\frac{\gamma}{2}} \, dt}_{(i)}
+ \underbrace{\int_0^1 \frac{(1-t^2)^{k}}{(1-2rt+r^2)^\frac{\gamma}{2}} \, dt}_{(ii)}
$$

Since $\gamma>0$ and $k+1>0$,
$$
 (i)  \le \int_{-1}^0 (1+t)^{k} \, dt = C
$$
\begin{align*}
 (ii) & \le \int_0^1 \frac{(1-t)^{k}}{(1-2rt+r^2)^{\frac{\gamma}{2}}} \, dt
\\ & = -\frac{1}{k+1} \int_0^1 \frac{\frac{d}{dt}[(1-t)^{k+1}]}{(1-2rt+r^2)^{\frac{\gamma}{2}}} \, dt
 \\ & = \frac{2r}{k+1} \int_0^1 \frac{(1-t)^{k+1}}{(1-2rt+r^2)^{\frac{\gamma}{2}+1}} \, dt 
 \\ & \le  C I_{\gamma+2, k+1}(r).
\end{align*}
and, since now $k+1>0$, $I_{\gamma, k}$ can be bounded as before. This concludes the proof of the lemma.
\end{proof}

Now we are ready to conclude the proof Theorem \ref{teo-convolucion}. Remember that we need to see that
\begin{equation}
\| x^{\frac{2\alpha+2}{q}-2b} I_{2(1+\alpha-\sigma),\alpha-\frac12}(x)
\|_{L^{s,\infty}(\frac{dx}{x})} < +\infty. 
\label{norma}
\end{equation}

Using the previous lemma, it is clear that when $x\to 1$ and $2(\alpha+1-\sigma)\le 2(\alpha - \frac12)$ the norm in \eqref{norma} is bounded.

In the case  $2(\alpha+1-\sigma)> 2(\alpha - \frac12)$ (that is, $\sigma<3$), the integrability condition is $$-s \left[ 2(\alpha+1-\sigma)-2 \left(\alpha-\frac12\right)-2 \right] \ge -1.$$
But, using \eqref{youngl}, we see that this is equivalent to $\sigma \ge -\frac12 \left( \frac{1}{q}-\frac{1}{p}\right)$, which holds by our assumption on $a+b$.

When $x=0$, the integrability condition is
$$
\frac{2\alpha+2}{q}-2b > 0
$$
which holds because $b<\frac{\alpha+1}{q}$.

Finally, when $x\to\infty$, since $I_{\alpha-\frac12, 2(\alpha+1-\sigma)}(x)\sim x^{-2(\alpha+1-\sigma)}$, the condition we need to fulfill is 
$$
\frac{2\alpha+2}{q}-2b-2(\alpha+1-\sigma)<0
$$
which, by our assumption on $a+b$ is equivalent to $a<\frac{\alpha+1}{p'}$.
\end{proof}

\subsection{Extension to the case  $-1<\alpha <0$ and end of proof of Theorem  \ref{main-result}}

As before, we may assume that $a+b= \left(\frac{1}{q}-\frac{1}{p}\right)(\alpha+1) +\sigma.$ In this case, to extend our result to the case  $-1<\alpha <0$ let us consider $-1<\alpha<\beta$, where $\beta\ge 0$, and use a transplantation result from  \cite{Garrigos}, that we recall here as a lemma for the sake of completeness:

\begin{lemma}[\cite{Garrigos}, Corollary  6.19 (ii)]
\label{lema-garrigos}
Let $1<q<\infty$. Given $\alpha, \beta >-1$, we define the transplantation operator
$$
\mathbb{T}_\beta^\alpha f = \sum_{k=0}^\infty \left( \int_0^\infty f(y) l_k^\alpha(y) y^{\alpha} \, dy \right) l_k^\beta.
$$
Then, if $\sigma_0 \in \mathbb{R}$ and $\sigma_1 = \sigma_0 + (\alpha -\beta)(\frac{1}{p} - \frac12)$, $\mathbb{T}_\beta^\alpha : L_{\sigma_0}^q (\mathbb{R}_+, x^\alpha \, dx) \to L_{\sigma_1}^q (\mathbb{R}_+, x^\beta \, dx)$ and $\mathbb{T}_\alpha^\beta : L_{\sigma_1}^q (\mathbb{R}_+, x^\beta \, dx) \to L_{\sigma_0}^q (\mathbb{R}_+, x^\alpha \, dx)$ are bounded operators if and only if 
$$
-\frac{1+\alpha}{q}<\sigma_0 <\frac{1+\alpha}{q^\prime}.
$$
\end{lemma}

Using this lemma, we can write 
\begin{align*}
\|M_{\alpha,m}f |x|^{-b}\|_{L^q(\mathbb{R}_+,x^{\alpha} \, dx)} & = \|\mathbb{T}_\alpha^\beta (M_{\beta,m}(\mathbb{T}_\beta^\alpha f)) |x|^{-b}\|_{L^q(\mathbb{R}_+, x^\alpha \, dx)}
\\ & \le C \|M_{\alpha,m, \beta}(\mathbb{T}_\beta^\alpha f) |x|^{-\tilde b}\|_{L^q(\mathbb{R}_+, x^{\beta} \, dx)}
\end{align*}
provided that
\begin{equation}
-1<\alpha<\beta
\end{equation}
\begin{equation}
\label{btilde}
-\tilde b = -b +(\alpha-\beta)\left(\frac{1}{q}-\frac12\right),
\end{equation}
and
\begin{equation}
\label{balfaq}
-\frac{1+\alpha}{q}<-b<\frac{1+\alpha}{q'},
\end{equation}
and, using Theorem \ref{teo-convolucion} for $M_{\beta,m}$ with $\beta \ge 0$, 
$$
\|M_{\alpha,m, \beta}(\mathbb{T}_\beta^\alpha f) |x|^{-\tilde b}\|_{L^q(\mathbb{R}_+, x^{\beta} \, dx)}
 \le C\|\mathbb{T}_\beta^\alpha f |x|^{\tilde a}\|_{L^p(\mathbb{R}_+, x^{\beta} \, dx)}
$$
provided that 
\begin{equation*}
0<\sigma<\beta+1 \quad, \quad \tilde a < \frac{\beta+1}{p'} \quad , \quad \tilde b <\frac{\beta+1}{q},
\end{equation*}
\begin{equation}
\label{desigtilde}
\left(\frac{1}{q}-\frac{1}{p}\right)\left(\beta+\frac12\right) \le \tilde a +\tilde b
\end{equation}
and that
\begin{equation}
\label{atildebtilde}
 \tilde a +\tilde b = \left(\frac{1}{q}-\frac{1}{p}\right)(\beta+1) + \sigma.
\end{equation}

Finally, using Lemma \ref{lema-garrigos}  again, we obtain
\begin{equation}
\|M_{\alpha,m}f |x|^{-b}\|_{L^q(\mathbb{R}_+,x^{\alpha} \, dx)} \le C \|f |x|^a\|_{L^p(\mathbb{R}_+, x^\alpha \, dx)}
\end{equation}
provided that
\begin{equation}
\label{atilde2}
\tilde a=a+(\alpha-\beta)\left(\frac{1}{p}-\frac{1}{2}\right)
\end{equation}
and that
\begin{equation}
\label{aalfap}
-\frac{1+\alpha}{p}<a<\frac{1+\alpha}{p'}.
\end{equation}

Now, replacing \eqref{btilde} and \eqref{atilde2} into \eqref{desigtilde} and \eqref{atildebtilde}  we obtain
\begin{equation*}
 \left(\frac{1}{q}-\frac{1}{p}\right)\left(\alpha+\frac12\right) \le a + b
 \end{equation*}
 and
 \begin{equation}
 \label{res}
 \quad a+b= \left(\frac{1}{q}-\frac{1}{p}\right)(\alpha+1) + \sigma.
\end{equation}
To conclude the proof of the theorem we need to see that the restrictions $a> -\frac{1+\alpha}{p}$ in \eqref{aalfap} and $b>-\frac{1+\alpha}{q'}$ in \eqref{balfaq} are redundant. Indeed, the first one follows from \eqref{res} and $b<\frac{\alpha+1}{q}$, while the second one follows from \eqref{res} and $a<\frac{\alpha+1}{p'}$. 

\section{Multipliers for related Laguerre systems}

In this section we show how the results for multipliers for expansions in the Laguerre system $\{l^\alpha_k\}_{k\ge 0}$  can be extended to other related systems, using a transference result from I. Abu-Falah, R. A. Mac\'ias, C. Segovia and J. L. Torrea \cite{AMST}. To this end, for fixed $\alpha>-1$, we consider the orthonormal systems:
\begin{enumerate}
\item $\{\mathcal{L}_k^\alpha(y) := y^{\frac{\alpha}{2}} l_k^\alpha(y)\}_{k\ge 0}$ in $L^2(\mathbb{R}_+)$
\item $\{\varphi_k^\alpha(y) := \sqrt 2 y^{\alpha+\frac12} l_k^\alpha(y^2)\}_{k\ge 0}$ in $L^2(\mathbb{R}_+)$
\item $\{\psi_k^\alpha (y) := \sqrt 2 l_k^\alpha(y^2)\}_{k\ge 0}$ in $L^2(\mathbb{R}_+, y^{2\alpha+1} \, dy)$
\end{enumerate}
which are eigenvectors of certain modifications of the Laguerre differential operator \eqref{laguerre}. 

Then, following the notations in \cite{AMST}, if we let $W^\alpha, V,$ and $Z^\alpha$ be the operators defined by
\begin{equation*}
W^\alpha f(y)=y^{-\frac{\alpha}{2}} f(y), \quad Vf(y)= (2y)^\frac12 f(y^2) , \quad and \quad Z^\alpha f(y)= \sqrt 2 y^{-\alpha} f(y^2)
\end{equation*}
it is immediate that $W^\alpha \mathcal{L}_k^\alpha = l_k^\alpha$, $V \mathcal{L}_k^\alpha = \varphi_k^\alpha$, and $Z^\alpha \mathcal{L}_k^\alpha = \psi_k^\alpha$. Moreover, for $f$ a measurable function with domain in $\mathbb{R}_+$, the following result holds:

\begin{lemma}[\cite{AMST}, Lemma 3.22]
\label{lema-cambio}
Let $\alpha>-1$. 
\begin{enumerate}
\item Let $\delta = \rho -\alpha(\frac{p}{2}-1)$, then $\|W^\alpha f \|_{L^p(\mathbb{R}_+, y^{\rho+\alpha})}= \|f \|_{L^p(\mathbb{R}_+, y^\delta)}$
\item Let $2\delta = \gamma +\frac{p}{2}-1$, then $\|Vf\|_{L^p(\mathbb{R}_+, y^\gamma)}= 2^{\frac12 -\frac{1}{p}}\|f \|_{L^p(\mathbb{R}_+,y^\delta)}$
\item Let $\delta = \frac{\eta}{2} -\alpha(\frac{p}{2}- 1)$, then $\|Z^\alpha f \|_{L^p(\mathbb{R}_+, y^{\eta+2\alpha+1})}= 2^{\frac12 -\frac{1}{p}}\|f\|_{L^p(\mathbb{R}_+, y^\delta)}$
\end{enumerate}
\end{lemma}

In analogy to what we have done for the system $\{l_k^\alpha\}_{k\ge 0}$, we can also define multipliers of Laplace transform type for the orthonormal systems listed above. For instance, in the case of the system  $\{\mathcal{L}_k^\alpha\}_{k\ge 0}$, if 
\begin{equation*}
 f(x) \sim \sum_{k=0}^\infty  b_{\alpha,k}(f) \mathcal{L}_k^\alpha(x), \quad   b_{\alpha,k}(f)= \int_0^\infty f(x) \mathcal{L}_k^\alpha(x)  dx
\end{equation*}
given a bounded sequence $\{m_k\}_{k\ge 0}$ we may define the multiplier
\begin{equation*}
 M_{\alpha,m}^\mathcal{L} f(x) \sim \sum_{k=0}^\infty  b_{\alpha,k}(f) m_k \mathcal{L}_k^\alpha(x),
 \end{equation*}
and we say that $M_{\alpha,m}^\mathcal{L}$ is a multiplier of Laplace transform type if $m_k=m(k)$ is given by \eqref{Laplace.transform} for some real-valued function $\Psi(t)$. Similar definitions can be given for the systems $\{\varphi_k^\alpha\}_{k\ge 0}$ and $\{\psi_k^\alpha\}_{k\ge 0}$; we will denote the corresponding multipliers by $M_{\alpha,m}^\varphi$ and $ M_{\alpha,m}^\psi$. Then, the following analogue of Theorem \ref{main-result} holds:

\begin{theorem}
\label{teo31}
Assume that  $\alpha>-1$. 
\begin{enumerate}
\item If $M_{\alpha,m}^{\mathcal{L}}$ is a multiplier of Laplace transform type for the system $\{\mathcal{L}_k^\alpha\}_{k\ge 0}$  such that $(H1)$ and $(H2)$ hold, then
\begin{equation*}
\|M_{\alpha,m}^{\mathcal{L}} f\|_{L^q(\mathbb{R}_+, x^{-Bq})} \le C \|f\|_{L^p(\mathbb{R}_+, x^{Ap})} 
\end{equation*}
provided that 
\begin{equation*}
1<p \le q<\infty \quad, \quad A< \frac{\alpha}{2}+\frac{1}{p'} \quad , \quad B<\frac{\alpha}{2}+\frac{1}{q},
\end{equation*}
 and that
\begin{equation*}
 \left( \frac{1}{q}-\frac{1}{p}\right)(\alpha+1)<A+B \le \sigma \left(\frac{1}{q}-\frac{1}{p} \right).
\end{equation*}

\item  If $M_{\alpha,m}^\varphi$ is a multiplier of Laplace transform type for the system $\{\varphi_k^\alpha\}_{k\ge 0}$  such that $(H1)$ and $(H2)$ hold, then
\begin{equation*}
\|M_{\alpha,m}^\varphi  f\|_{L^q(\mathbb{R}_+, x^{-Dq})} \le C \|f\|_{L^p(\mathbb{R}_+, x^{Cp})} 
\end{equation*}
provided that 
\begin{equation*}
1<p \le q<\infty \quad, \quad C< \alpha + \frac{1}{p'}+\frac12 \quad , \quad D< \alpha+\frac{1}{q}+\frac12
\end{equation*} 
and that
\begin{equation*}
 \left(\frac{1}{q}-\frac{1}{p}\right)(2\alpha+1) <C+D\le (2\sigma-1) \left(\frac{1}{q}-\frac{1}{p}\right).
\end{equation*}

\item If $M_{\alpha,m}^\psi$ is a multiplier of Laplace transform type for the system $\{\psi_k^\alpha\}_{k\ge 0}$  such that $(H1)$ and $(H2)$ hold, then
\begin{equation*}
\|M_{\alpha,m}^\psi f\|_{L^q(\mathbb{R}_+, x^{-Fq})} \le C \|f\|_{L^p(\mathbb{R}_+, x^{Ep})} 
\end{equation*}
provided that 
\begin{equation*} 
1<p \le q<\infty \quad, \quad E< 2\alpha + 1+ \frac{1}{p'} \quad , \quad F< \frac{1}{q}
\end{equation*} and that
\begin{equation*}
 \left(\frac{1}{q}-\frac{1}{p}\right)(2\alpha+1) <E+F\le (2\sigma-1) \left(\frac{1}{q}-\frac{1}{p}\right).
\end{equation*}

\end{enumerate}
\end{theorem}
\begin{proof}
We explain how to prove (1), since the other cases are analogous. From the fact that $W^\alpha \mathcal{L}_k^\alpha =l_k^\alpha$ and by Lemma \ref{lema-cambio}(1), we have the following diagram

\begin{equation*}
\begin{array}{cccc}
 & L^p(\mathbb{R}_+, x^{ap+\alpha}) &\stackrel{M_{\alpha,m}}\longrightarrow & L^q(\mathbb{R}_+ , x^{-bq+\alpha})
 \\   & (W^\alpha)^{-1}  \Big\downarrow & &  \Big\uparrow W^\alpha
 \\ &  L^p(\mathbb{R}_+, x^{Ap}) & \stackrel{M_{\alpha,m}^\mathcal{L}}\longrightarrow  & L^q(\mathbb{R}_+ , x^{-Bq})
\end{array}
\end{equation*}
provided that 
\begin{equation}
\label{condicionAa}
Ap = ap - \alpha \left( \frac{p}{2}-1 \right) \quad \mbox{and}Ê\quad -Bq = -bq - \alpha\left( \frac{q}{2}-1\right).
\end{equation}
and $M_{\alpha, m} = W^\alpha M_{\alpha,m}^\mathcal{L} (W^\alpha)^{-1}$. Therefore, the identities \eqref{condicionAa} together with the  conditions on $a,b$ given by Theorem \ref{main-result} imply the desired result.
\end{proof}

\section{Proof of Theorem \ref{teorema-hermite}}

In this section we exploit the well-known relation between  Hermite and Laguerre poynomials to obtain an analogous result to that of Section 2  in the Hermite case.  Indeed, recalling that 
\begin{align*} 
H_{2k}(x) & = (-1)^k 2^{2k} k! L_k^{-\frac12} (x^2)
\\ H_{2k+1}(x) & = (-1)^k 2^{2k} k! x L_k^{\frac12} (x^2) 
\end{align*}
it is immediate that
\begin{align*}
h_{2k}(x) &= l^{-1/2}_k(x^2)
\\ h_{2k+1}(x) & = x l^{\frac12}_k(x^2)
\end{align*}

It is then natural to decompose $f=f_0+f_1$ where 
$$ f_0(x)=\frac{f(x)+f(-x)}{2} \quad , \quad f_1(x)= \frac{f(x)-f(-x)}{2} $$
and, clearly, when $k=2j$, if we let  $g_0(y)= f_0(\sqrt{y})$ we obtain:
$$ c_k(f)  = \langle f_0, h_k \rangle =  2 \int_0^\infty f_0(x) l^{-\frac12}_j(x^2) \; dx = a_{-\frac12,j}(g_0) $$
while if $k=2j+1$, and we let $g_1(y)= \frac{1}{\sqrt{y}} f_1(\sqrt{y})$ we have:
 $$ c_k(f) =  \langle f_1, h_k \rangle  = 2 \int_0^\infty f_1(x) x  l^{\frac12}_j(x^2) \; dx = a_{\frac12,j}(g_1) $$

Then,
\begin{align*}
M_{H,m} f (x) & = \sum_{j=0}^\infty m_{2j} a_{-\frac12,j}(g_0 ) l^{-\frac12}_j(x^2) + \sum_{j=0}^\infty m_{2j+1} a_{\frac12,j}( g_1) x l^{\frac12}_j(x^2)
\\ & = M_{-\frac12,m_0} g_0(x^2) + x M_{\frac12, m_1} g_1(x^2)
\end{align*}
where $(m_0)_k=m_{2k}$ and $(m_1)_k = m_{2k+1}$.

To apply Theorem \ref{main-result} to this decomposition, we need to check  first that $m_0$ and $m_1$ are Laplace-Stiltjes functions of certain functions $\Psi_0$ and $\Psi_1$. Indeed, notice that $m_{2k}= \mathfrak{L}\Psi_0(k)$ where 
$$\Psi_0(u)=\frac12 \Psi(\frac{u}{2})$$
 and 
$m_{2k+1}= \mathfrak{L} \Psi_1(k)$ where 
$$\Psi_1(u)= \frac12  \int_0^\frac{u}{2} e^{-\tau} d\Psi(\tau).$$

It is also easy to see that $\Psi_0$ satisfies the hypotheses of Theorem \ref{main-result} for $\alpha=-\frac12$ whereas  $\Psi_1$ satisfies the hypotheses for $\alpha=\frac12$ (in this case condition $(H2)$ follows after an integration by parts). 

Then,
\begin{align}
\nonumber \|M_{H,m}f |x|^{-b}\|_{L^q(\mathbb{R})} & = \left( \int_{\mathbb{R}} |M_{H,m}f(x)|^q |x|^{-bq} \, dx\right)^{\frac{1}{q}}
\\ & = C \left( \int_{\mathbb{R}} \left|M_{-\frac12,m_0} g_0(x^2) +  x M_{\frac12, m_1} g_1(x^2) \right|^q |x|^{-bq} \, dx\right)^{\frac{1}{q}} \label{Mh}
\end{align}

Using Minkowski's inequality and making the change of variables $y=x^2, dx= \frac12 y^{-\frac12} \, dy$, we see that
\begin{align*}
\eqref{Mh} &\sim \left( \int \left|M_{-\frac12,m_0} g_0(y)\right|^q |y|^{-\frac{bq}{2}-\frac12} \, dy \right)^{\frac{1}{q}} +  \left( \int  \left|M_{\frac12, m_1} g_1(y)\right|^q |y|^{\frac{(-b+1)q}{2}-\frac12} \, dy\right)^{\frac{1}{q}}
\\ & = \| M_{-\frac12,m_0} g_0(y) |y|^{-\frac{b}{2}}\|_{L^q(\mathbb{R}, x^{-\frac12} \, dx)} + \|M_{\frac12, m_1} g_1(y) |y|^{\frac{-b+1}{2}-\frac{1}{q}} \|_{L^q(\mathbb{R}, x^\frac12 \, dx)}
\\ &\le C \|g_0(y) |y|^{\tilde a}\|_{L^p(\mathbb{R}, x^{-\frac12} \, dx)} + C  \|g_1(y) |y|^{\hat a}\|_{L^p(\mathbb{R}, x^{\frac12} \, dx)}
\end{align*}
where the last inequality follows from Theorem  \ref{main-result} provided that:
 \begin{equation*}
 \tilde a < \frac{1}{2p'} \quad , \quad b < \frac{1}{q}
 \end{equation*}
 \begin{equation}
 \label{resc1}
0\le \tilde a +\frac{b}{2} \le \frac12 \left(\frac{1}{q}-\frac{1}{p} \right)+\sigma
 \end{equation}
 \begin{equation*}
 \hat a < \frac{3}{2p'} 
  \end{equation*}
  and
 \begin{equation}
 \label{resc2}
\left( \frac{1}{q}-\frac{1}{p} \right) \le \hat a + \frac{1}{q} -\frac{1-b}{2} \le \frac32 \left( \frac{1}{q}-\frac{1}{p}\right)+ \sigma.
\end{equation}
Therefore,
\begin{align*}
\|M_{H,m}f |x|^{-b}\|_{L^q(\mathbb{R})}  & \le C \left( \int |g_0(x)|^p |x|^{\tilde a p-\frac12} \, dx \right)^{\frac{1}{p}} + C \left( \int |g_1(x)|^p |x|^{\hat a p + \frac12} \, dx\right)^{\frac{1}{p}}
\\ &= C \left( \int |f_0(\sqrt{x})|^p |x|^{\tilde a p-\frac12} \, dx \right)^{\frac{1}{p}} + C \left( \int |f_1(\sqrt{x})|^p |x|^{\hat a p + \frac12-\frac{p}{2}} \, dx\right)^{\frac{1}{p}}
\\ & = C \left( \int |f_0(x)|^p |x|^{2\tilde a p} \, dx \right)^{\frac{1}{p}} + C \left( \int |f_1(x)|^p |x|^{2\hat a p + 2-p} \, dx\right)^{\frac{1}{p}}
\\ & \le C \|f(x) |x|^a\|_{L^p(\mathbb{R})}
\end{align*}
provided that
\begin{equation}
\label{aatildeahat}
a = 2\tilde a = 2 \hat a + \frac{2}{p}-1.
\end{equation}

Therefore, by \eqref{aatildeahat} and the conditions on $ \tilde a, \hat a$, there must hold
$$
a<\frac{1}{p'}
$$
while, by  \eqref{aatildeahat},  \eqref{resc1} and \eqref{resc2} are equivalent to 
$$
0 \le a+b \le \frac{1}{q} - \frac{1}{p}+ 2\sigma.
$$

\begin{remark}
It follows from the proof of Theorem \ref{teorema-hermite} that a better result holds if the function $f$ is odd. 
\end{remark}

\section{Examples and further remarks}

First, we should point out that it is clear  that, since a Stieltjes integral of a continuous function with respect to a
function of bounded variation can be thought as an integral with respect to the
corresponding Lebesgue-Stieltjes measure, we could equivalently have formulated all our results in terms of integrals with respect to signed Borel measures in $\R_+$. However, we
have found convenient to use the framework of
Stieltjes integrals since many of the classical references on Laplace transforms are written in that
framework (for instance \cite{Widder}), and leave the details of a possible restatement of the theorems in the case of regular Borel measure to the reader.

We also recall that the Laplace-Stieltjes transform contains as particular cases both the ordinary Laplace transform of (locally integrable) functions (when $\Psi(t)$ is absolutely continuous), and Dirichlet series (see  below). In particular, if $\Psi$ is absolutely continuous and
$\phi(t)=\Psi^\prime(t)$ (defined almost everywhere), the assumptions $(H1)$ and $(H2)$ of Theorem \ref{main-result} can be replaced by:
\begin{itemize}
\item[(H1ac)] $$ \int_0^\infty |\phi(x)| \; dx < +\infty \quad
\hbox{i.e.} \; \phi \in L^1(\R_+) $$
\item[(H2ac)] there exist $\delta>0$,  $0 < \sigma < \alpha+1$, and $C>0$ such that
$$ \left|\int_0^t \phi(x) \; dx \right| \leq C t^{\sigma} \quad \hbox{for} \; 0 < t \leq \delta. $$
\end{itemize}
In particular, assumption $(H2ac)$ holds if $\phi(t)=O(t^{\sigma-1})$ when $t \to 0$.

As we have already mentioned in the introduction, B. Wr\'obel \cite[Corollary 2.7]{W} has recently proved that Laplace type multipliers for the system $\{\varphi_k^\alpha\}_{k\ge 0}$ are  bounded on $L^p(\mathbb{R}^d,\omega)$, $1<p<\infty$, for all $\omega \in A_p$ and $\alpha \in (\{-\frac12\} \cup [\frac12, \infty))^d$. In the case of power weights in one dimension this means that $\omega(x)=|x|^\beta$ must satisfy $-1<\beta<p-1$, while  taking $p=q$ and letting the weight be $|x|^\beta$ on both sides, Theorem \ref{teo31}(2) can easily be seen to imply $-1-p\left(\alpha+\frac12 \right)<\beta<p-1+ p\left(\alpha +\frac12\right)$. 

Also, weighted estimates had been obtained before for the case of some particular operators for the system $\{l_k^\alpha\}_{k\ge 0}$. Indeed, recall that  one of the main examples of the kind of multipliers we are considering is the 
Laguerre  fractional integral introduced in \cite{GST}, which corresponds to the choice
$m_k=(k+1)^{-\sigma}$. 

In \cite[Theorem 4.2]{Nowak-Stempak}, A. Nowak and K.  Stempak considered multi-dimensional Laguerre expansions and used
a slightly different definition of the fractional integral operator,  given by the negative powers of the differential operator  \eqref{laguerre}.

 As they point out,  their theorem contains as a special case the result of \cite{GST} (in the one dimensional case). To  see that both operators are indeed equivalent, they rely on a deep multiplier theorem \cite[Theorem 1.1]{Stempak-Trebels}.

Instead, we can see that Theorem \ref{main-result} is applicable to both definitions by choosing:
$$ m_k=(k+c)^{-\sigma}, \quad \phi(t)= \frac{1}{\Gamma(\sigma)} t^{\sigma-1} e^{-ct} \quad
(c>0) $$
 The case $c=1$ corresponds to the definition in \cite{GST}, whereas the choice $c=\frac{\alpha+1}{2}$ corresponds to the definition in \cite{Nowak-Stempak}. Therefore, Theorem \ref{main-result} applied to these choices, coincides in the first case with the result of \cite[Theorem 1]{GT} (which is an improvement of \cite[Theorem 3.1]{GST}) and improves in the second case the one-dimensional result of \cite[Theorem 4.2]{Nowak-Stempak}.
 
 The same choice of $m_k$ and $\phi$ in Theorem \ref{teorema-hermite} gives a two-weight estimate for the Hermite fractional integral, which corresponds to the one-dimensional version of  \cite[Theorem 2.5]{Nowak-Stempak}. 

Another interesting example is the operator $(L^2+I)^{-\frac{\alpha}{2}}$, where $L$ is given by \eqref{laguerre}. In this case, Theorem \ref{main-result}  with hypotheses $(H1ac)$ and $(H2ac)$ instead of $(H1)$ and $(H2)$ applies with $\alpha= \sigma$ and 
$$
\phi(t)= \frac{1}{C_\alpha} e^{-\frac{\alpha+1}{2}t}J_{\frac{\alpha-1}{2}} (t) t^{\frac{\alpha-1}{2}}
$$
since, by \cite[formula 5, p. 386]{Watson},
$$
\int_0^\infty e^{-st} J_{\frac{\alpha -1}{2}}(t) t^{\frac{\alpha-1}{2}} \, dt = C_\alpha (s^2+1)^{-\frac{\alpha}{2}}
$$
and, when $t\to 0$, $J_{\frac{\alpha-1}{2}} (t) t^{\frac{\alpha-1}{2}} \sim t^{\alpha-1}$.

A further  example is obtained by choosing
$\Psi(t)=e^{-s_0t} H(t-\tau)$  
with $s_0=\frac{\alpha+1}{2}$, where $H$ is the Heaviside unit step function:
$$ 
H(t) = \left\{
\begin{array}{rcl}
1 & \hbox{if} & t \geq 0 \\
0 & \hbox{if} & t < 0 \\
\end{array} \right. 
$$
and we see that Theorem \ref{main-result} is applicable to the Heat diffusion semigroup (considered for instance in \cite{Stempak-heat-diffusion} and \cite{MST}) 
$$ M_{\tau} = e^{-\tau L} $$
 associated to the operator $L$ for any $\sigma>0$. More generally, the same conclusion holds for 
$$\Psi(t)= \sum_{n=1}^\infty a_n e^{-s_0t}  H(t-\tau_n )  $$
 provided that the Dirichlet series
$$ F(s)= \sum_{n=1}^\infty a_n e^{-\tau_n s}, \quad 0 < \tau_1 < \tau_2 < \ldots $$
 conveges absolutely for $s=s_0$ (which corresponds to
hypothesis $(H1)$).

As a final comment, we remark that finding a function $\Psi$  of bounded variation such that  $m_k = \mathfrak{L} \Psi(k)$  holds (see \eqref{Laplace.transform}) is equivalent to solving the clasical Hausdorff moment problem (see \cite[Chapter III]{Widder}).

\medskip

{\bf Acknowledgements.} We wish to thank Professor K. Stempak for bringing into our attention the connection between the generalized euclidean convolution and our previous results on fractional integrals of radially symmetric functions, and for helpful comments and corrections. 

We are also indebted to Professor J. L. Torrea for pointing to us that our results  could be transferred from one Laguerre system to the others, and for giving us reference \cite{AMST}.

\end{document}